\documentclass[a4paper]{amsart}
\usepackage{mathrsfs}
\usepackage{bbm}
\usepackage{latexsym, amssymb,amsmath,amsthm}
\usepackage[all, knot]{xy}
\xyoption{arc}

\theoremstyle{plain}
\def \To{\longrightarrow}
\def \dim{\operatorname{dim}}
\def \Q{\operatorname{Q}}
\def \P{\operatorname{P}}
\def \S{\operatorname{S}}
\def \B{\operatorname{B}}
\def \Soc{\operatorname{Soc}}
\def \id{\operatorname{id}}
\def \End{\operatorname{End}}

\def \q{\mathbbm{q}}
\def \dim{\operatorname{dim}}

\def \k{\kappa}

\def \D{\Delta}

\def \e{\varepsilon}

\def\Quslz{{\Q}\mathbf{u}_{q}(\mathfrak{sl}_{2})}
\def\uqslz{\mathbf{u}_{q}(\mathfrak{sl}_{2})}
\def\Uqslz{\mathfrak{U}_{q}(\mathfrak{sl}_{2})}

\numberwithin{equation}{section}

\newtheorem{theorem}{Theorem}[section]
\newtheorem{lemma}[theorem]{Lemma}
\newtheorem{proposition}[theorem]{Proposition}
\newtheorem{corollary}[theorem]{Corollary}

\newtheorem{remark}[theorem]{Remark}

\begin{document}
\title{Representations of the small quasi-quantum group $\Quslz$}
\author{Gongxiang Liu}
\address{Department of Mathematics, Nanjing University, Nanjing 210093, China} 
\email{gxliu@nju.edu.cn}
\author{Fred Van Oystaeyen}
\address{Department of Mathematics and Computer Science, University of Antwerp, Antwerp, Belgium}
 \email{fred.vanoystaeyen@ua.ac.be}
\author{Yinhuo Zhang}
\address{Department of Mathematics and Statistics, University of Hasselt, Hasselt, Belgium}
\email{yinhuo.zhang@uhasselt.be}

\begin{abstract}  The quasi-Frobenius-Lusztig kernel $\Quslz$
associated with $\mathfrak{sl}_{2}$ has been constructed in \cite{Liu}. In this paper we study the representations of this small quasi-quantum group. We give a complete list of non-isomorphic indecomposables and the tensor product decomposition rules for simples and projectives. A description of the Grothendieck ring is provided.
\end{abstract}
\maketitle

\section{Introduction}
In \cite{Liu} the first author introduced a quasi-Hopf version of the small quantum group $\uqslz$ and denoted it by $\Quslz$, where $q$ is an $n^2$-th primitive root of unity for some natural number $n$. It is proved in loc.cit. (see also Lemma 2.1 here) that for odd $n$, $\Quslz$ is twisted equivalent  to the Hopf algebra $\uqslz$, but for even $n$ it is not! So for even $n$ the $\Quslz$ is a new quasi-Hopf algebra. The purpose of this paper is to study the representations of this new algebra.  We restrict, for convenience, to the case $4|n$ while the general case $2|n$ will be remarked.

The representation theory of the small quantum group $\uqslz$ and the restricted quantum universal enveloping algebra $\mathfrak{U}_{q}(\mathfrak{sl}_{2})$ associated to $\mathfrak{sl}_{2}$  have been studied extensively, cf. \cite{Lus,Su,Xiao,KS}.
So for us the problem to study the finite dimensional representations of $\Quslz$ arises naturally. We will provide in this paper a complete list of non-isomorphic indecomposable modules. Here a new phenomenon appears: $\Quslz$ has no Steinberg modules (i.e. simple projective modules). Moreover, the dimensions of all simple modules are \emph{odd}. Furthermore, $\Quslz$ has some interesting ``symmetry" properties: all the blocks have the same dimension and they are Morita equivalent to one-another. Another new phenomenon of interest appears here:  the basic algebra of $\Quslz$ can be equipped with a Hopf algebra structure, whereas we can prove that this does not happen for $\uqslz$ and $\Uqslz$.

To understand $\Quslz$ further, it is helpful to study the decomposition rules of tensor products of modules, i.e.,  a version of the Clebsch-Gordan formula for $\Quslz$. After proving that the direct summands of the tensor products of two simple modules are either simple or projective, the decomposition rules for the tensor products of simples and projectives, as well as simples and simples, are given explicitly. From this, the structure of its Grothendieck ring is derived. 

The paper is organized as follows.  In Section 2, we describe all simple modules and projective modules of  $\Quslz$, and consequently the basic algebra of each block may be given by using quiver and relations. In Proposition \ref{prop2.15}, we establish a Hopf algebra structure on the basic algebra of $\Quslz$, and prove in Proposition \ref{prop2.16}  that this is not the case for $\uqslz$ and $\Uqslz$. 

Section 3 is devoted to finding a complete list of non-isomorphic indecomposable modules (Theorem \ref{thm3.1}). Finally in Section 4, we provide the decomposition rules of tensor products (Theorem \ref{thm4.8}) and determine the Grothendieck ring $K_0$ of $\Quslz$ (Theorem \ref{thm4.9}). Thus we obtained  a fairly complete description of the representation theory of $\Quslz$, highlighting new interesting phenomena when compared to the existing theory for $\uqslz$ and $\Uqslz$.

Throughout this paper, we work over a fixed algebraically closed field $k$ of characteristic zero.
\section{Simples, projectives and basic algebra }
 We recall the definition of $\Quslz$ from \cite{Liu}. Let $n$ be a natural number and
 $q$ an $n^{2}$-th primitive root of unity. Let $\q: =q^{n}$. The quasi-Hopf algebra
 $\Quslz$ is defined as follows. As an associative algebra, it is generated by four
    elements $\k,\hat{\k},E,F$ satisfying
    \begin{gather} \k^{n}=1,\;\;\hat{\k}^{n}=\k^{-2},\;\;\k\hat{\k}=\hat{\k}\k,\\
    \k E\k^{-1}=\q E,\;\;\;\;\k F\k^{-1}=\q^{-1} F,\\
    \hat{\k}E\hat{\k}^{-1}=\q q^{-2}E,\;\;\;\;\hat{\k}F\hat{\k}^{-1}=\q^{-1}q^{2}F,\\
    E^{n^{2}}=F^{n^{2}}=0,\\
    FE-q^{-1}EF=1-\k^{-1}\hat{\k}.
    \end{gather}
 Define
    \begin{equation}1_{i}:=\frac{1}{n}\sum_{j=0}^{n-1}(\q^{n-i})^{j}\k^{j},\;\;\;\;\flat=\sum_{i=0}^{n-1}q^{-i}1_{i}.\end{equation}
    The reassociator $\phi_{s}$, the comultiplication $\D$, the counit $\e$, the elements $\alpha,\beta$ and the antipode $S$
    are given as follows:
    \begin{gather} \phi_{s}=\sum_{i,j,k=0}^{n-1}\q^{-i[\frac{j+k}{n}]}1_{i}\otimes
    1_{j}\otimes 1_{k},\\
    \D(\k)=\k\otimes \k,\;\;\;\;\D(\hat{\k})=\hat{\k}\otimes \hat{\k},\\
    \D(E)=E\otimes \flat^{-1}+ \k^{-1}\otimes 1_{0}E+1\otimes \sum_{i=1}^{n-1}1_{i}E,\\
    \D(F)=F\otimes \flat+\k^{-1}\hat{\k}\otimes F\sum_{i=1}^{n-1}1_{i}+\hat{\k}\otimes F1_{0},\\
    \e(\k)=\e(\hat{\k})=1,\;\;\;\;\e(E)=\e(F)=0,\\
    \alpha=\k,\;\;\;\;\beta=1\\
    S(\k)=\k^{-1},\;\;\;\;S(\hat{\k})=\hat{\k}^{-1},\\
    S(x)=-(\k\sum_{i=1}^{n-1}1_{i}E+\k^{2}1_{0}E)\flat \k^{-1},\\
    \;\;S(F)=-(\k^{2}\hat{\k}^{-1}F\sum_{i=1}^{n-1}1_{i}
    +\k\hat{\k}^{-1}F1_{0})\flat^{-1}\k^{-1}.
    \end{gather}

    Combining \cite[Thm. 4.3]{EG} and \cite[Thm. 4.1]{Liu}, we have the following.
    
\begin{lemma}\label{l2.1} \emph{(1)} If $n$ is odd, then $\Quslz$ is twist equivalent to the Hopf algebra $\uqslz$.

\emph{(2)} If $n$ is even, then $\Quslz$ is not twist equivalent to
    any Hopf algebra.    
\end{lemma}

By Lemma \ref{l2.1}, we only need to consider the case where $n$ is even. For convenience, we assume $4|n$ throughout the paper. The results for $2|n$ case will be stated as remarks.

Let $\mathbf{u}^{+},\mathbf{u}^{-}$ and $\mathbf{u}^{0}$ be the subalgebras of $\Quslz$ generated by
$E,F$ and $\{\k,\hat{\k}\}$ respectively. Then $\Quslz$ has a triangle decomposition
$$\Quslz=\mathbf{u}^{-}\mathbf{u}^{0}\mathbf{u}^{+}.$$
Let $m=\frac{n}{2}$. Define
$$\varphi_{2i}:=\frac{1}{mn}\sum_{k=0}^{mn-1}q^{-2ik}(\k^{-1}\hat{\k})^{k},\;\;\mathbbm{1}_{0}:=\frac{1}{2}(1+\k\hat{\k}^{\frac{n}{2}}),\;\;
\mathbbm{1}_{1}:=\frac{1}{2}(1-\k\hat{\k}^{\frac{n}{2}}).$$
Let $e_{2i,0}=\varphi_{2i}\mathbbm{1}_{0}$ and $e_{2i,1}=\varphi_{2i}\mathbbm{1}_{1}$.

\begin{lemma}\label{l2.2} The set
$\{e_{2i,0},e_{2i,1}|1\leq i\leq \frac{n^{2}}{2}\}$ is a complete set of primitive idempotents of $\mathbf{u}^{0}$.
\end{lemma}
\begin{proof} Let $V$ be the space spanned by $e_{2i,0},e_{2i,1}$ for $1\leq i\leq \frac{n^{2}}{2}$. It is enough to show that $\k$ and $\hat{\k}$ both
belong to $V$. By the definition of $e_{2i,0},e_{2i,1}$, we know $\k^{-1}\hat{\k},\k\hat{\k}^{\frac{n}{2}}\in V$ and thus $\hat{\k}^{\frac{n}{2}+1}\in V$. Clearly, the order of $\hat{\k}$ is $\frac{n^{2}}{2}$. To show $\hat{\k}\in V$, it suffices to show that $\frac{n}{2}+1$ is coprime to $\frac{n^{2}}{2}$. From the identity: $(\frac{n}{2}-1)(\frac{n}{2}+1)=\frac{n^{2}}{4}-1$, we know that $l.c.m(\frac{n}{2}+1,\frac{n^{2}}{2})=1$ or $2$. If
$l.c.m(\frac{n}{2}+1,\frac{n^{2}}{2})=2$, we have $2|(\frac{n}{2}+1)$ which is absurd by our assumption that $4|n$. Therefore, $\hat{\k}\in V$
and hence $\k\in V$.
\end{proof}

\begin{remark}\label{r2.3} \emph{If we only assume that $2|n$, we can not assure that Lemma \ref{l2.2} is always true. However, if we define
$\varphi'_{2i}:=\frac{1}{mn}\sum_{k=0}^{mn-1}q^{-2ik}\hat{\k}^{k}$ and $e'_{2i,0}:=\varphi'_{2i}\mathbbm{1}_{0},e'_{2i,1}:=\varphi'_{2i}\mathbbm{1}_{1}$.
Then one can show that the set $\{e'_{2i,0},e'_{2i,1}|1\leq i\leq \frac{n^{2}}{2}\}$ is always a complete set of primitive idempotents of $\mathbf{u}^{0}$.} \end{remark}

\begin{lemma}\label{l2.3} The following identities hold in $\Quslz$.
\begin{gather*} \k^{-1}\hat{\k}E=q^{-2}E\k^{-1}\hat{\k},\;\;
\k^{-1}\hat{\k}F=q^{2}F\k^{-1}\hat{\k},\\
\k\hat{\k}^{\frac{n}{2}}E=-E\k\hat{\k}^{\frac{n}{2}},\;\;
\k\hat{\k}^{\frac{n}{2}}F=-F\k\hat{\k}^{\frac{n}{2}},\\
\k^{-1}\hat{\k}e_{2i,0}=q^{2i}e_{2i,0},\;\;\k^{-1}\hat{\k}e_{2i,1}=q^{2i}e_{2i,1},\\
\k\hat{\k}^{\frac{n}{2}}e_{2i,0}=e_{2i,0},\;\;\k\hat{\k}^{\frac{n}{2}}e_{2i,1}=-e_{2i,1}.
\end{gather*}
\end{lemma}
\begin{proof} Straightforward.
\end{proof}
For any natural number $s$, define $s_{q}:=1+q+\cdots +q^{s-1}$.
\begin{lemma}\label{l2.4} The following identities hold in $\Quslz$.
\begin{gather} FE^{s}=q^{-s}E^{s}F+s_{q^{-1}}E^{s-1}-s_{q}\k^{-1}\hat{\k}E^{s-1},\label{eq;2.1}\\
E^{s}F=q^{s}FE^{s}+q s_{q^{-1}}E^{s-1}\k^{-1}\hat{\k}-q s_{q}E^{s-1},\\
EF^{s}=q^{s}F^{s}E+q^{s}s_{q}F^{s-1}\k^{-1}\hat{\k}-q s_{q}F^{s-1},\\
F^{s}E=q^{-s}EF^{s}+s_{q^{-1}}F^{s-1}-s_{q}F^{s-1}\k^{-1}\hat{\k}.
\end{gather}
\end{lemma}
\begin{proof} We only prove the first one since the other proofs are similar. It is clear that the formula \eqref{eq;2.1} is true for $s=1$. Now assume that the formula \eqref{eq;2.1} holds for $s$. We show that it holds for $s+1$. Indeed, we have:
\begin{eqnarray*}
FE^{s+1}&=&(q^{-s}E^{s}F+s_{q^{-1}}E^{s-1}-s_{q}\k^{-1}\hat{\k}E^{s-1})E\\
&=&q^{-s}E^{s}(q^{-1}EF+(1-\k^{-1}\hat{\k}))+s_{q^{-1}}E^{s}-s_{q}\k^{-1}\hat{\k}E^{s}\\
&=&q^{-(s+1)}E^{s+1}F +q^{-s}E^{s}-q^{-s}q^{2s}\k^{-1}\hat{\k}E^{s}+s_{q^{-1}}E^{s}-s_{q}\k^{-1}\hat{\k}E^{s}\\
&=&q^{-(s+1)}E^{s+1}F+(s+1)_{q^{-1}}E^{s}-(s+1)_{q}\k^{-1}\hat{\k}E^{s}.
\end{eqnarray*}
\end{proof}
For $1\leq i\leq \frac{n^{2}}{2}$, define \begin{equation}
\alpha_{2i,0}:=F^{n^{2}-1}e_{2i,0},\;\;\;\;\alpha_{2i,1}:=F^{n^{2}-1}e_{2i,1}.
\end{equation}
\begin{corollary}\label{c2.1} For $1\leq i\leq \frac{n^{2}}{2}$, $j=0,1$ and $s\leq n^{2}-1$, we have:
\begin{gather*}FE^{s}\alpha_{2i,j}=s_{q^{-1}}(1-q^{2i-1-s})E^{s-1}\alpha_{2i,j},\;\;\k^{-1}\hat{\k}E^{s}\alpha_{2i,j}=
q^{2i-2-2s}E^{s}\alpha_{2i,j},\\
\;\;\k\hat{\k}^{\frac{n}{2}}E^{s}\alpha_{2i,j}=(-1)^{s+1+j}E^{s}\alpha_{2i,j}.
\end{gather*}
\end{corollary}
\begin{proof} These are direct consequences of Formula \eqref{eq;2.1} and Lemma \ref{l2.3}.
\end{proof}

Thus, for any $1\leq i\leq \frac{n^{2}}{2}$ and $j=0,1$, the $n^{2}$-dimensional left ideal $\Quslz\alpha_{2i,j}$
 may be represented schematically as:

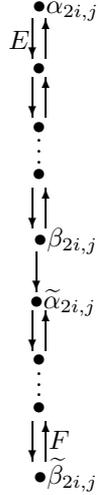
\begin{figure}[hbt]
\begin{picture}(150,200)(80,25)
\put(100 ,128){\makebox(0,0){$\bullet \small{\beta_{2i,j}}$}}
\put(92.5,133){\vector(0,1){15}} \put(87.5,148){\vector(0,-1){15}}
\put(90,153){\makebox(0,0){$\bullet$}}
\put(90,165){\makebox(0,0){$\vdots$}}
 \put(100,215){\makebox(0,0){$\bullet \small{\alpha_{2i,j}}$}}
\put(92.5,197){\vector(0,1){15}}
\put(87.5,212){\vector(0,-1){15}}\put(82.5,204){\makebox(0,0){$E$}}
\put(90,193){\makebox(0,0){$\bullet$}}
\put(92.5,175){\vector(0,1){15}} \put(87.5,190){\vector(0,-1){15}}
\put(90,171){\makebox(0,0){$\bullet$}} \put(89,
124){\vector(0,-1){15}} \put(99,105){\makebox(0,0){$\bullet
\small{\widetilde{\alpha}_{2i,j}}$}}
\put(100 ,40){\makebox(0,0){$\bullet
\small{\widetilde{\beta}_{2i,j}}$}}
\put(92.5,45){\vector(0,1){15}}\put(97.5,53){\makebox(0,0){$F$}}
\put(87.5,60){\vector(0,-1){15}}
\put(90,65){\makebox(0,0){$\bullet$}}
\put(90,77){\makebox(0,0){$\vdots$}}
\put(92.5,87){\vector(0,1){15}} \put(87.5,102){\vector(0,-1){15}}
\put(90,83){\makebox(0,0){$\bullet$}}
\end{picture}
\end{figure}

where $\beta_{2i,j}=E^{2i-2}\alpha_{2i,j},\widetilde{\alpha}_{2i,j}=E^{2i-1}\alpha_{2i,j}$ and
 $\widetilde{\beta}_{2i,j}=E^{n^{2}-1}\alpha_{2i,j}$. Each dot stands for a 1-dimensional subspace and
an upward (resp. downward) arrow indicates a nonzero left-multiplication action by $F$ (resp. $E$). We call $\Quslz\alpha_{2i,j}$
a \emph{Verma} module.

Contrasting to the classical $\uqslz$ case (e.g. see page 362 in \cite{Su}), the single downward arrow \emph{always} appears since $2i-1$ is odd while $n^{2}$ is even. Therefore, we always have $F^{n^{2}-1}\beta_{2i,j}=0$ for $1\leq i\leq \frac{n^{2}}{2}$ and $j=0,1$, and
\begin{equation} \beta_{2i,j}=F\gamma_{2i,j},
\end{equation}
for a unique vector $\gamma_{2i,j}$ of the form $\gamma_{2i,j}=\gamma_{2i,j}^{-}\gamma_{2i,j}^{0}\gamma_{2i,j}^{+}$
for $\gamma_{2i,j}^{\star}\in \mathbf{u}^{\star}$ with $\star=-,0,+$. Observe  that $\gamma_{2i,j}$ and $\widetilde{\alpha}_{2i,j}$ have
the same left $\k^{-1}\hat{\k}$- and $\k\hat{\k}^{\frac{n}{2}}$-eigenvalues.

\begin{lemma}\label{l2.5} For $1\leq i\leq \frac{n^{2}}{2},j=0,1$ and $s\leq n^{2}-1$, we have
$$FE^{s}\gamma_{2i,j}-q^{-s}E^{s}\beta_{2i,j}=s_{q^{-1}}(1-q^{1-s-2i})E^{s-1}\gamma_{2i,j}.$$
\end{lemma}
\begin{proof} By Lemma \ref{l2.4}, we have:
 \begin{eqnarray*}
FE^{s}\gamma_{2i,j}&=&(q^{-s}E^{s}F+s_{q^{-1}}E^{s-1}-s_{q}\k^{-1}\hat{\k}E^{s-1})\gamma_{2i,j}\\
&=& q^{-s}E^{s}\beta_{2i,j}+s_{q^{-1}}E^{s-1}\gamma_{2i,j}-q^{-(s-1)}s_{q^{-1}}E^{s-1}\k^{-1}\hat{\k}\gamma_{2i,j}\\
&=&q^{-s}E^{s}\beta_{2i,j}+s_{q^{-1}}E^{s-1}\gamma_{2i,j}-q^{-(s-1)}s_{q^{-1}}q^{-2i}E^{s-1}\gamma_{2i,j}\\
&=&q^{-s}E^{s}\beta_{2i,j}+s_{q^{-1}}(1-q^{1-s-2i})E^{s-1}\gamma_{2i,j}.
\end{eqnarray*}
\end{proof}

\begin{corollary}\label{c2.7} For $1\leq i\leq \frac{n^{2}}{2}$ and $j=0,1$, $F^{n^{2}-1}E^{n^{2}-1}\gamma_{2i,j}$ is a nonzero multiple of
 $\widetilde{\alpha}_{2i,j}$.
\end{corollary}
\begin{proof} Using Lemma \ref{l2.5} repeatedly, we have:
 \begin{eqnarray*}
F^{n^{2}-1}E^{n^{2}-1}\gamma_{2i,j}&=&\prod_{t=n^{2}-1}^{n^{2}-2i}t_{q^{-1}}(1-q^{1-t-2i})F^{n^{2}-2i+1}E^{n^{2}-2i+1}\gamma_{2i,j}\\
&=&q^{-(n^{2}-2i+1)}\prod_{t=n^{2}-1}^{n^{2}-2i}t_{q^{-1}}(1-q^{1-t-2i})F^{n^{2}-2i}E^{n^{2}-2i+1}F\gamma_{2i,j}\\
&=&q^{-(n^{2}-2i+1)}\prod_{t=n^{2}-1}^{n^{2}-2i}t_{q^{-1}}(1-q^{1-t-2i})F^{n^{2}-2i}\widetilde{\beta}_{2i,j}.
 \end{eqnarray*}
 Since $F^{n^{2}-2i}\widetilde{\beta}_{2i,j}$ is clearly a nonzero multiple of $\widetilde{\alpha}_{2i,j}$, the lemma is proved.
\end{proof}

\begin{corollary}\label{c2.9} For $1\leq i\leq \frac{n^{2}}{2}$ and $j=0,1$, the vectors in $\{E^{l}\alpha_{2i,j},E^{l}\gamma_{2i,j}|0\leq l\leq n^{2}-1\}$
are linear independent.
\end{corollary}
\begin{proof} Corollary \ref{c2.7} entails that all vectors in $\{E^{l}\alpha_{2i,j},E^{l}\gamma_{2i,j}|0\leq l\leq n^{2}-1\}$
are nonzero.  By definition, $\gamma_{2i,j}$ and $E^{2i-1}\alpha_{2i,j}$ are linear independent. From this, we deduce that all vectors
in  $\{E^{s}\alpha_{2i,j},E^{t}\gamma_{2i,j}|0\leq t\leq n^{2}-2i,\;2i-1\leq s\leq n^2-1\}$ are linear independent which implies the desired result
 since the other vectors have different heights (see the beginning of
Section 4 for the definition of height).
\end{proof}

For $1\leq i\leq \frac{n^{2}}{2}$ and $j=0,1$, the left ideal $\Quslz\gamma_{2i,j}$
may be represented schematically as:\newpage

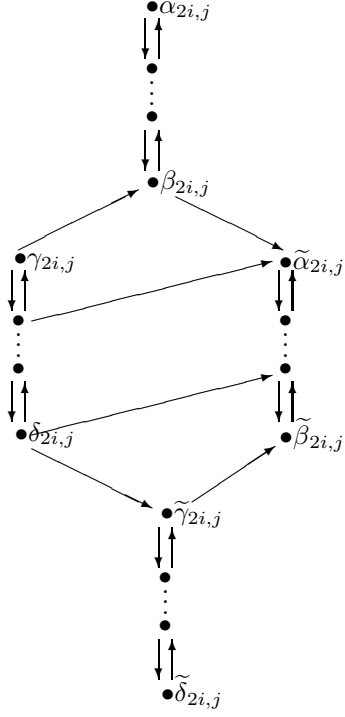
\begin{figure}[hbt]
\begin{picture}(200,265)(0,-50)
\put(100 ,150){\makebox(0,0){$\bullet \small{\beta_{2i,j}}$}}
\put(92.5,155){\vector(0,1){15}} \put(87.5,170){\vector(0,-1){15}}
\put(90,175){\makebox(0,0){$\bullet$}}
\put(90,187){\makebox(0,0){$\vdots$}} \put(100
,215){\makebox(0,0){$\bullet \small{\alpha_{2i,j}}$}}
\put(92.5,197){\vector(0,1){15}} \put(87.5,212){\vector(0,-1){15}}
\put(90,193){\makebox(0,0){$\bullet$}}

\put(50,120){\makebox(0,0){$\bullet \gamma_{2i,j}$}}
\put(40,125){\vector(2,1){45}} \put(99,145){\vector(2,-1){40}}
\put(150,120){\makebox(0,0){$\bullet \widetilde{\alpha}_{2i,j}$}}

\put(50 ,55){\makebox(0,0){$\bullet \small{\delta_{2i,j}}$}}
\put(42.5,60){\vector(0,1){15}} \put(37.5,75){\vector(0,-1){15}}
\put(40,80){\makebox(0,0){$\bullet$}}
\put(40,92){\makebox(0,0){$\vdots$}}
\put(42.5,102){\vector(0,1){15}} \put(37.5,117){\vector(0,-1){15}}
\put(40,98){\makebox(0,0){$\bullet$}}

\put(45,98){\vector(4,1){90}} \put(45,55){\vector(4,1){90}}

\put(150 ,55){\makebox(0,0){$\bullet
\small\widetilde{\beta}{_{2i,j}}$}} \put(142.5,60){\vector(0,1){15}}
\put(137.5,75){\vector(0,-1){15}}
\put(140,80){\makebox(0,0){$\bullet$}}
\put(140,92){\makebox(0,0){$\vdots$}}
\put(142.5,102){\vector(0,1){15}}
\put(137.5,117){\vector(0,-1){15}}
\put(140,98){\makebox(0,0){$\bullet$}}

\put(105,25){\makebox(0,0){$\bullet \widetilde{\gamma}_{2i,j}$}}
\put(45,50){\vector(2,-1){45}} \put(105,30){\vector(3,2){31}}

\put(105 ,-42){\makebox(0,0){$\bullet
\small\widetilde{\delta}{_{2i,j}}$}} \put(97.5,-37){\vector(0,1){15}}
\put(92.5,-22){\vector(0,-1){15}}
\put(95,-17){\makebox(0,0){$\bullet$}}
\put(95,-5){\makebox(0,0){$\vdots$}}
\put(97.5,5){\vector(0,1){15}} \put(92.5,20){\vector(0,-1){15}}
\put(95,1){\makebox(0,0){$\bullet$}}
\end{picture}
\end{figure}

where $\delta_{2i,j}=E^{n^{2}-2i}\gamma_{2i,j}, \widetilde{\gamma}_{2i,j}=E^{n^{2}-2i+1}\gamma_{2i,j}$ and
$\widetilde{\delta}_{2i,j}=E^{n^{2}-1}\gamma_{2i,j}$. Put
\begin{equation} \P_{2i,j}:=\Quslz\gamma_{2i,j},
\end{equation}
and
\begin{equation} \S_{2i,j}:=\Soc(\P_{2i,j}),
\end{equation}
the socle of $\P_{2i,j}$. That is, $\S_{2i,j}$ can be represented as:

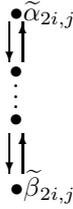
\begin{figure}[hbt]
\begin{picture}(200,80)(50,0)
\put(150,70){\makebox(0,0){$\bullet \widetilde{\alpha}_{2i,j}$}}
\put(150 ,5){\makebox(0,0){$\bullet
\small\widetilde{\beta}{_{2i,j}}$}}
\put(142.5,10){\vector(0,1){15}}
\put(137.5,25){\vector(0,-1){15}}
\put(140,30){\makebox(0,0){$\bullet$}}
\put(140,42){\makebox(0,0){$\vdots$}}
\put(142.5,52){\vector(0,1){15}}
\put(137.5,67){\vector(0,-1){15}}
\put(140,48){\makebox(0,0){$\bullet$}}
\end{picture}
\end{figure}

It is a simple $\Quslz$-module  isomorphic to the top of $\P_{2i,j}$.

\begin{lemma}\label{l2.6} For $1\leq i,i'\leq \frac{n^{2}}{2}$ and $j,j'=0,1$, $\S_{2i,j}\cong \S_{2i',j'}$ if and only if
$i=i'$ and $j=j'$.
\end{lemma}
\begin{proof} It it not hard to see that $\dim \S_{2i,j}=n^{2}-2i+1$. If $\S_{2i,j}\cong \S_{2i',j'}$, then $\dim \S_{2i,j}
=\dim \S_{2i',j'}$ and so $i=i'$. By comparing the $\k\hat{\k}^{\frac{n}{2}}$-eigenvectors, we have $j=j'$.
\end{proof}

\begin{lemma}\label{l2.7} For $1\leq i\leq \frac{n^{2}}{2}$ and $j=0,1$, we have in $\Quslz$:
$$E^{n^{2}-1}\alpha_{2i,j}E^{n^{2}-2i}\neq 0.$$
\end{lemma}
\begin{proof} By using the fourth formula in Lemma \ref{l2.4} repeatedly, we have:
\begin{eqnarray*}
E^{n^{2}-1}\alpha_{2i,j}E^{n^{2}-2i}&=&E^{n^{2}-1}F^{n^{2}-1}E^{n^{2}-2i}e_{\overline{n^{2}-2i},j}\\
&=&E^{n^{2}-1}(q^{-(n^{2}-1)}EF^{n^{2}-1}+(n^{2}-1)_{q^{-1}}F^{n^{2}-2}-(n-1)_{q}F^{n^{2}-2}\k^{-1}\hat{\k})\\
&&E^{n^{2}-2i-1}e_{\overline{n^{2}-2i},j}\\
&=&E^{n^{2}-1}((n^{2}-1)_{q^{-1}}-(n^{2}-1)_{q}q^{2i+2})F^{n^{2}-2}E^{n^{2}-2i-1}e_{\overline{n^{2}-2i},j}\\
&=&(n^{2}-1)_{q^{-1}}(1-q^{2i})E^{n^{2}-1}F^{n^{2}-2}E^{n^{2}-2i-1}e_{\overline{n^{2}-2i},j}\\
&=&\cdots\\
&=&\prod_{t=1}^{n^{2}-2i}(n^{2}-t)_{q^{-1}}(1-q^{t-1+2i})E^{n^{2}-1}F^{2i-1}e_{\overline{n^{2}-2i},j},
\end{eqnarray*}
where $\overline{i}$ denotes the least positive residue of $i$ modulo $n^{2}$. It is not hard to see that $E^{n^{2}-1}F^{2i-1}e_{\overline{n^{2}-2i},j}\neq 0$.
\end{proof}

\begin{corollary}\label{c2.8} The right multiplication by $E^{h}$ defines an isomorphism $\P_{2i,j}\stackrel{\cong}{\To} \P_{2i,j}E^{h}$
for $0\leq h\leq n^{2}-2i$.
\end{corollary}
\begin{proof} It is enough to show that the right multiplication by $E^{n^{2}-2i}$ is a monomorphism, but this is a direct consequence of Lemma \ref{l2.7}.
\end{proof}

\begin{theorem}\label{t2.9} As a left $\Quslz$-module, we have:
$$\Quslz=\bigoplus_{j=0}^{1}\bigoplus_{ i=1}^{\frac{n^{2}}{2}}\bigoplus_{h=0}^{n^{2}-2i}\P_{2i,j}E^{h}.$$
\end{theorem}
\begin{proof} By counting the dimensions of both sides, we only need to show that the sum $\sum_{i,j,h}\P_{2i,j}E^{h}$ is a direct sum.\\[2mm]
\textbf{Claim.} $\P_{2i,j}E^{h} \bigcap \P_{2i',j'}E^{h'}\neq 0$ if and only if $i=i',j=j'$ and $h=h'$.

\emph{Proof of this claim.} If $\P_{2i,j}E^{h} \bigcap \P_{2i',j'}E^{h'}\neq 0$, then they have the same socle. Since $\P_{2i,j}\cong \P_{2i,j}E^{h}$
and $\P_{2i',j'}\cong \P_{2i',j'}E^{h'}$, the socles of $\P_{2i,j}$ and $\P_{2i',j'}$ are isomorphic. Therefore, Lemma \ref{l2.6} implies
$i=i'$ and $j=j'$. Consequently, $h=h'$.

Inductively, we assume that the sum of any $n$ terms of $\P=\{\P_{2i,j}E^{h_{i}}|1\leq i\leq \frac{n^{2}}{2},0\leq h_{i}\leq n^{2}-2i,0\leq j\leq 1\}$
is a direct sum. We show the conclusion for $n+1$ terms. Take $M_{1},\ldots,M_{n+1}\in \P$ and assume that $\sum_{l=1}^{n+1}M_{l}$ is not
a direct sum. Then there is a $l$, say $n+1$, such that $M_{n+1}\bigcap \sum_{l=1}^{n}M_{l}=M_{n+1}\bigcap \bigoplus_{l=1}^{n}M_{l} \neq 0$.
Therefore, $\Soc (M_{n+1})=\Soc (M_{l})$ for some $1\leq l\leq n$ and thus $M_{n+1}\bigcap M_{l}\neq 0$, which is absurd by the Claim.
\end{proof}

Now we have the following conclusions.
\begin{corollary}\label{Cor2.14} 
\begin{enumerate}
\item[(a)] $\{\P_{2i,j}|1\leq i\leq \frac{n^{2}}{2},j=0,1\}$ forms a complete set of non-isomorphic indecomposable projective $\Quslz$-modules.
\item[(b)] $\{\S_{2i,j}|1\leq i\leq \frac{n^{2}}{2},j=0,1\}$ forms a complete set of non-isomorphic simple $\Quslz$-modules.
\item[(c)] $\Quslz$ has $n^{2}$  non-isomorphic indecomposable projective modules and every indecomposable projective module is has dimension $2n^{2}$.
\item[(d)] $\Quslz$ has $n^{2}$  non-isomorphic simple modules and every simple module is of odd dimension.
\item[(e)] $\P_{2i,j}$ and $\P_{2i',j'}$ belong to the same block if and only if $2i+2i'=n^{2}+2$ and $j+j'=1$.
\item[(f)] The number of blocks of $\Quslz$ is $\frac{n^{2}}{2}$  and each block has dimension $2n^{4}$. Moreover, every block of $\Quslz$ is Morita equivalent to the following basic algebra $\Lambda$: 
\begin{figure}[hbt]
\begin{picture}(200,60)(0,0)

 \put(0,30){\makebox(0,0){$\cdot$}}
  \put(5,32){\vector(1,0){30}} \put(20,35){\makebox(0,0){$x_{1}$}}
  \put(5,42){\vector(1,0){30}}\put(20,45){\makebox(0,0){$y_{1}$}}

  \put(40,30){\makebox(0,0){$\cdot$}} \put(35,28){\vector(-1,0){30}} \put(20,22){\makebox(0,0){$x_{2}$}}
  \put(35,18){\vector(-1,0){30}}\put(20,12){\makebox(0,0){$y_{2}$}}

\put(80,35){\makebox(0,0){$x_{s}x_{t}=y_{s}y_{t}$}}
\put(180,30){\makebox(0,0){$\textrm{for}\; 1\leq s\neq t\leq 2.$}}
\put(90,25){\makebox(0,0){$x_{s}y_{t}=y_{s}x_{t}=0$}}
\end{picture}
\end{figure}
\item[(g)] $\Quslz$ is of tame representation type.
\end{enumerate}
\end{corollary}
\begin{proof} (a) is a direct consequence of Theorem \ref{t2.9}. (a) implies (b) since every projective module is also injective (This follows from
the fact that every finite-dimensional quasi-Hopf algebra is Frobenius). By Corollary 2.9, $\dim \P_{2i,j}=2n^{2}$ and thus we obtain (c).
Statement (d) is clear since $\dim \S_{2i,j}=n^{2}-2i+1$.

Now let $J$ be the Jacobson radical of $\Quslz$. Then we have the following isomorphisms:
\begin{eqnarray*}\P_{2i,0}/J\P_{2i,0}\cong \S_{2i,0},&\;\;\;\;&\P_{2i,1}/J\P_{2i,1}\cong \S_{2i,1},\\
J\P_{2i,0}/J^2\P_{2i,0}\cong \S_{n^2-2i+2,1}\oplus \S_{n^2-2i+2,1},&\;\;\;\;&J\P_{2i,1}/J^2\P_{2i,1}\cong \S_{n^2-2i+2,0}\oplus \S_{n^2-2i+2,0},\\
\Soc(\P_{2i,0})=J^{2}\P_{2i,0}\cong \S_{2i,0},&\;\;\;\;& \Soc (\P_{2i,1})=J^{2}\P_{2i,1}\cong \S_{2i,1}.
\end{eqnarray*}
These imply that $\P_{2i,0}$ and $\P_{n^{2}-2i+2,1}$ belong to the same block for $1\leq i\leq\frac{n^{2}}{2}$. Therefore, Statement (e) follows.

(c)$+$(e) implies that the number of blocks of $\Quslz$ is $\frac{n^{2}}{2}$. Denote by $\B_{2i,j}$ the block containing $\P_{2i,j}$. The representation theory of finite-dimensional algebras tells us  that the dimension of $\B_{2i,0}$ is equal to
$$\dim \S_{2i,0}\dim \P_{2i,0}+\dim \S_{n^2-2i+2,1}\dim \P_{n^2-2i+2,1}=2n^{2}[(n^2-2i+1)+(2i-1)]=2n^{4}.$$
Moreover, the basic algebra of $\B_{2i,0}$ is isomorphic to the opposite algebra of $$\End_{\Quslz}(\P_{2i,0}\oplus \P_{n^{2}-2i+1,1}).$$
 Parallel to \cite[Sec. 5]{Su}, we can easily show that $\End_{\Quslz}(\P_{2i,0}\oplus \P_{n^{2}-2i+1,1})\cong \Lambda$.
 But the opposite algebra of $\Lambda$ is isomorphic to itself. Hence, (f) is proved.

Note that the basic algebra $\Lambda$ was studied extensively, see for example \cite{Erd,Ha,Su,Xiao}. It is known that $\Lambda$  is a tame algebra and thus we obtain the last statement (g).
\end{proof}

Let $\zeta_{l}$ be an $l$-th primitive
root of unity and $m$ a positive integer satisfying $(m,l)=1$. Denote by 
$\textbf{h}(\zeta_{l},m)$ the algebra $k\langle y,x,g\rangle/(x^{l},y^{l},g^{l}-1,gx-\zeta_{l}xg,gy-\zeta_{l}^{m}yg,xy-yx)$. This algebra $\textbf{h}(\zeta_{l},m)$ can be equipped with a Hopf algebra structure
 with comultiplication, antipode and counit given by 
$$\Delta(x)=x\otimes g+1\otimes x,\;\;\;\;\Delta(y)=y\otimes 1+g^{m}\otimes y,\;
\;\;\;\Delta(g)=g\otimes g$$
$$S(x)=-xg^{-1},\;\;S(y)=-g^{-m}y,\;\;S(g)=g^{-1},\;\;\varepsilon(x)=
\varepsilon(y)=0.\;\;\varepsilon(g)=1.$$ 
Note that such a Hopf algebra $\textbf{h}(\zeta_{l},m)$
is called a \emph{book algebra} in \cite{AH}. It is a basic algebra since
$\textbf{h}(q,m)/J_{\textbf{h}(q,m)}$ is a commutative semisimple
algebra. 

\begin{proposition}\label{prop2.15} The basic subalgebra $\B(\Quslz)$ of $\Quslz$ has a Hopf algebra structure such that
$$\B(\Quslz)\cong \textbf{h}(-1,1)\otimes k\mathbbm{Z}_{\frac{n^{2}}{2}}\ \ \ {\rm as\ Hopf\ algebras}.$$
\end{proposition}
\begin{proof} By Corollary \ref{Cor2.14} (f), we have an algebra isomorphism: 
$$\B(\Quslz)\cong \Lambda^{(\frac{n^{2}}{2})}.$$
It is not hard to see that $\Lambda \cong \textbf{h}(-1,1)$. This implies that the basic algebra of each block can be equipped with
a Hopf algebra structure, i.e., the book algebra $\textbf{h}(-1,1)$. Therefore, we have an isomorphism of Hopf algebras:
$$\B(\Quslz)\cong \textbf{h}(-1,1)\otimes k\mathbbm{Z}_{\frac{n^{2}}{2}}.$$
\end{proof}

Contrasting to this, we have the following.
\begin{proposition}\label{prop2.16} There is no Hopf algebra
structure on the basic algebras $\B(\uqslz)$ and  $\B(\Uqslz)$. 
\end{proposition}
\begin{proof} Let $H$ be a finite-dimensional Hopf algebra such that the underlying algebra is basic. By \cite[Thm. 2.3]{GS}, the
Ext-quiver of $H$ must be a \emph{covering quiver} (see \cite{GS} for the definition), or equivalently a \emph{Hopf quiver} (see \cite{CR}).
By the definition of the covering quiver, we know that the Ext-quivers of all blocks of $H$ are isomorphic as direct graphs.

 Now let
$H=\uqslz$ or $H=\mathfrak{U}_{q}(\mathfrak{sl}_{2})$. Then it is well-known that $H$ contains a Steinberg module. 
Therefore, the block $\B_{s}$ containing this Steinberg module is Morita equivalent to $k$. Thus its Ext-quiver is just a point. 
So if there is a Hopf structure on $\B(H)$, then all blocks are simple algebras by the foregoing argument. It follows that $H$ is semisimple, which is absurd. 
\end{proof}

\begin{remark} \emph{(a) In the classical $\uqslz$ case or the restricted quantum universal enveloping algebra $\Uqslz$ case, the order of the group-like element $K$ (see \cite{Lus,Su} for the definitions of $\uqslz$  and $\Uqslz$) is high enough to distinguish between different vectors in an indecomposable projective module and differentiate one projective module from another. However, in the $\Quslz$ case, we lose this convenient tool partly because the orders of the group-like elements $\k$ and $\hat{\k}$ are not high enough.
Fortunately, we can still use them to differentiate two non-isomorphic projective modules.}

\emph{(b) If $2|n$ while $4\nmid n$, one can use $e'_{2i,0},e'_{2i,1}$ defined in Remark
\ref{r2.3} and the same procedure we developed to define projective modules
 and simple modules. Moreover all conclusions in Corollary \ref{Cor2.14}  still hold. We leave the proof for the interested reader. }

 \emph{(c) There are two apparent differences between the representations of
 $\Quslz$ and those of the classical $\uqslz$ and $\Uqslz$. Namely,
 (I): $\Quslz$ has no Steinberg modules; (II): all simple modules of $\Quslz$ are of odd dimensions. On the other hand, it is well-known that both $\uqslz$ and $ \Uqslz$  have Steinberg modules and the dimensions of the simple modules may be even.}
\end{remark}

\section{Indecomposable modules}

Let $\Lambda$ be the basic algebra given in Corollary \ref{Cor2.14} (f). The Auslander-Reiten quiver $\Gamma_{\Lambda}$ of $\Lambda$ is known. Doubling the following picture  we obtain $\Gamma_{\Lambda}$:
\begin{figure}[hbt]
\begin{picture}(220,80)(0,-10)
\put(-10,40){\makebox(0,0){$\ddots$}}

\put(0,30){\makebox(0,0){$\cdot$}}
 \put(2,35){\vector(1,1){15}}
\put(4,33){\vector(1,1){15}}
\put(22,52){\makebox(0,0){$\cdot$}}\put(24,54){\vector(1,0){15}}
\put(45,54){\makebox(0,0){$\cdot$}}\put(50,54){\vector(1,0){15}}
\put(45,60){\makebox(0,0){$P$}}
 \put(24,48){\vector(1,-1){15}}
\put(28,52){\vector(1,-1){15}} \put(45,30){\makebox(0,0){$\cdot$}}

\put(45,30){\makebox(0,0){$\cdot$}}
 \put(47,35){\vector(1,1){15}}
\put(49,33){\vector(1,1){15}} \put(67,52){\makebox(0,0){$\cdot$}}
 \put(69,48){\vector(1,-1){15}}
\put(73,52){\vector(1,-1){15}} \put(90,30){\makebox(0,0){$\cdot$}}

\put(90,30){\makebox(0,0){$\cdot$}}
 \put(92,35){\vector(1,1){15}}
\put(94,33){\vector(1,1){15}} \put(112,52){\makebox(0,0){$\cdot$}}
\put(122,47){\makebox(0,0){$\ddots$}}

\put(170,70){\makebox(0,0){$\vdots$}}

\put(170,60){\makebox(0,0){$\cdot$}} \put(168,55){\vector(0,-1){15}}
\put(170,35){\makebox(0,0){$\cdot$}} \put(172,40){\vector(0,1){15}}

\put(170,35){\makebox(0,0){$\cdot$}} \put(168,30){\vector(0,-1){15}}
\put(170,10){\makebox(0,0){$\cdot$}} \put(172,15){\vector(0,1){15}}

\put(200,70){\makebox(0,0){$\vdots$}}

\put(200,60){\makebox(0,0){$\cdot$}} \put(198,55){\vector(0,-1){15}}
\put(200,35){\makebox(0,0){$\cdot$}} \put(202,40){\vector(0,1){15}}

\put(200,35){\makebox(0,0){$\cdot$}} \put(198,30){\vector(0,-1){15}}
\put(200,10){\makebox(0,0){$\cdot$}} \put(202,15){\vector(0,1){15}}

\put(220,35){\makebox(0,0){$\cdots$}}
\put(150,35){\makebox(0,0){$\cdots$}}

\put(170,-10){\makebox(0,0){$\textrm{\Small A}\;
\mathbbm{P}^{1}k\; \textrm{\Small family of homogeneous
tubes}$}}
\end{picture}
\end{figure}

To give all indecomposables, we only need to construct indecomposable $\Quslz$-modules corresponding  to dots in $\Gamma_{\Lambda}$ by Corollary \ref{Cor2.14}(f). The construction is parallel to \cite[Sec. 4]{Xiao} and \cite[Sec. 5]{Su}. So we omit the proof here and state the results directly.

\subsection{The indecomposable modules $V^{2i,j}_{l}$.}
For any non-negative integer $l$ and $1\leq i \leq \frac{n^{2}}{2},j=0,1$, the indecomposable module $V^{2i,j}_{l}$ has a basis:
$$\{a_{u}(m-1),e_{v}(m)|0\leq m\leq l,\;0\leq u\leq n^{2}-2i,\;1\leq v\leq 2i-1\}$$
with the action given by
\begin{eqnarray*}
&&\k^{-1}\hat{\k}e_{v}(m)=q^{2i-2v}e_{v}(m),\;\;\k\hat{\k}^{\frac{n}{2}}e_{v}(m)=(-1)^{v+j}e_{v}(m)\\
&&E e_{v}(m)=e_{v+1}(m),\\
&&Fe_{v}(m)=(v-1)_{q^{-1}}(1-q^{2i-v})e_{v-1}(m)+\delta_{v,1}a_{n^{2}-2i}(m-1)\end{eqnarray*}
and
\begin{eqnarray*}
&&\k^{-1}\hat{\k}a_{u}(m-1)=q^{-2i-2u}a_{u}(m-1),\;\;\k\hat{\k}^{\frac{n}{2}}a_{u}(m-1)=(-1)^{u+j}a_{u}(m-1)\\
&&E a_{u}(m-1)=a_{u+1}(m-1),\\
&&F a_{u}(m-1)=(u+2i-1)_{q^{-1}}(1-q^{-u})a_{u-1}(m-1),
\end{eqnarray*}
where $a_{n^{2}-2i+1}(m-1)=a_{-1}(m-1)=a_{u}(-1)=e_{0}(m)=0$ and $e_{2i}(m)=a_{0}(m)$. It may be useful to depict this module by means
of diagram:

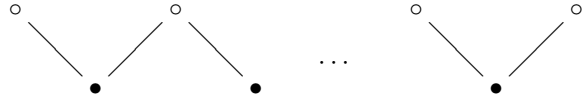
\begin{figure}[hbt]
\begin{picture}(200,50)(0,0)
\put(-30,40){\makebox(0,0){$\circ$}}
 \put(-25,35){\line(1,-1){20}}
\put(0,10){\makebox(0,0){$\bullet$}}
\put(5,15){\line(1,1){20}}
\put(30,40){\makebox(0,0){$\circ$}}
\put(35,35){\line(1,-1){20}}
\put(60,10){\makebox(0,0){$\bullet$}}

\put(90,20){\makebox(0,0){$\cdots$}}
\put(120,40){\makebox(0,0){$\circ$}}
 \put(125,35){\line(1,-1){20}}
\put(150,10){\makebox(0,0){$\bullet$}}
\put(155,15){\line(1,1){20}}
\put(180,40){\makebox(0,0){$\circ$}}

\end{picture}
\end{figure}

There are $l$ copies of $\bullet$ and $l+1$ copies of $\circ$.  In comparing with the diagrams displayed earlier, $\bullet$ and $\circ$ stand for $\S_{2i,j}$ and $\S_{n^{2}-2i+2,j'}$ respectively where
$j'+j=1$. The lines $/$ and $\backslash$ may be understood as the actions of $E$ and $F$ respectively.

They form all syzygies of simple modules. Indeed, we have:
$$\Omega^{l}(\S_{2i,j})=\left \{
\begin{array}{ll} V^{2i,j}_{l}, & \;\;\;\;l \textrm{ is odd}\\
V^{n^2-2i+2,j'}_{l}, &
\;\;\;\;l \textrm{ is even,}
\end{array}\right. $$
where $j'$ is determined by requiring  $j'+j=1$. Here $\Omega^{l}(\S_{2i,j})$ denotes the $l$-th syzygy of $\S_{2i,j}$.

\subsection{The indecomposable modules $\widetilde{V}^{2i,j}_{n}$.}
For any non-negative integer $l$ and $1\leq i \leq \frac{n^{2}}{2},j=0,1$, the indecomposable modules $\widetilde{V}^{2i,j}_{n}$ has a basis:
$$\{a_{u}(m+1),e_{v}(m)|0\leq m\leq l,\;0\leq u\leq n^{2}-2i,\;1\leq v\leq 2i-1\}$$
with the action given by
\begin{eqnarray*}
&&\k^{-1}\hat{\k}e_{v}(m)=q^{2i-2v}e_{v}(m),\;\;\k\hat{\k}^{\frac{n}{2}}e_{v}(m)=(-1)^{v+j}e_{v}(m),\\
&&E e_{v}(m)=e_{v+1}(m),\\
&&Fe_{v}(m)=(v-1)_{q^{-1}}(1-q^{2i-v})e_{v-1}(m),\end{eqnarray*}
and
\begin{eqnarray*}
&&\k^{-1}\hat{\k}a_{u}(m+1)=q^{-2i-2u}a_{u}(m+1),\;\;\k\hat{\k}^{\frac{n}{2}}a_{u}(m+1)=(-1)^{u+j}a_{u}(m+1),\\
&&E a_{u}(m+1)=a_{u+1}(m+1),\\
&&F a_{u}(m+1)=(u+2i-1)_{q^{-1}}(1-q^{-u})a_{u-1}(m+1)+\delta_{u,0}e_{2i-1}(m),
\end{eqnarray*}
where $a_{-1}(m)=a_{u}(0)=e_{-1}(m)=e_{2i}(m)=0$ and $a_{n^2-2i+2}(m)=e_{1}(m)$. This module can be described schematically as follows:
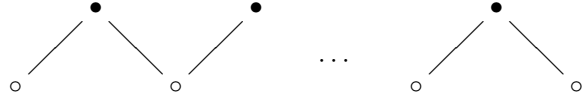
\begin{figure}[hbt]
\begin{picture}(200,50)(0,0)
\put(0,10){\makebox(0,0){$\circ$}}
\put(5,15){\line(1,1){20}}
\put(30,40){\makebox(0,0){$\bullet$}}
\put(35,35){\line(1,-1){20}}
\put(60,10){\makebox(0,0){$\circ$}}
\put(65,15){\line(1,1){20}}
\put(90,40){\makebox(0,0){$\bullet$}}

\put(120,20){\makebox(0,0){$\cdots$}}
\put(150,10){\makebox(0,0){$\circ$}}
\put(155,15){\line(1,1){20}}

\put(180,40){\makebox(0,0){$\bullet$}}
\put(185,35){\line(1,-1){20}}
\put(210,10){\makebox(0,0){$\circ$}}

\end{picture}
\end{figure}

There are $l$ copies of $\bullet$ and $l+1$ copies of $\circ$. The relation with cosyzygies is:
$$\Omega^{-l}(\S_{2i,j})=\left \{
\begin{array}{ll} \widetilde{V}^{2i,j}_{l} & \;\;\;\;l \textrm{ is odd}\\
\widetilde{V}^{n^2-2i+2,j'}_{l} &
\;\;\;\;l \textrm{ is even,}
\end{array}\right. $$
where $j'$ is determined by the requirement  $j'+j=1$. Here $\Omega^{-l}(\S_{2i,j})$ denotes the $l$-th cosyzygy of $\S_{2i,j}$.

\subsection{The indecomposable modules $W^{2i,j}_{l}$.}
For any positive integer $l$ and $1\leq i \leq \frac{n^{2}}{2},j=0,1$, one has a basis of $W^{2i,j}_{l}$ as follows:
$$\{e_{u}(m)|1\leq m\leq l,\;1\leq u\leq n^{2}\}$$
with the action given by
\begin{eqnarray*}
&&\k^{-1}\hat{\k}e_{u}(m)=q^{2i-2u}e_{u}(m),\;\;\k\hat{\k}^{\frac{n}{2}}e_{u}(m)=(-1)^{u+j}e_{u}(m),\\
&&E e_{u}(m)=e_{u+1}(m),\\
&&Fe_{u}(m)=(u-1)_{q^{-1}}(1-q^{2i-u})e_{u-1}(m)+\delta_{u,1}e_{n^{2}}(m-1),\end{eqnarray*}
where $e_{n^2+1}(m)=e_{0}(m)=e_{u}(0)=0$. The diagram of this module is:

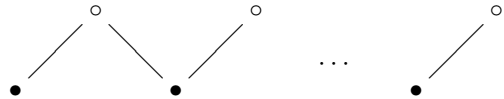
\begin{figure}[hbt]
\begin{picture}(200,50)(0,0)
\put(0,10){\makebox(0,0){$\bullet$}}
\put(5,15){\line(1,1){20}}
\put(30,40){\makebox(0,0){$\circ$}}
\put(35,35){\line(1,-1){20}}
\put(60,10){\makebox(0,0){$\bullet$}}
\put(65,15){\line(1,1){20}}
\put(90,40){\makebox(0,0){$\circ$}}

\put(120,20){\makebox(0,0){$\cdots$}}
\put(150,10){\makebox(0,0){$\bullet$}}
\put(155,15){\line(1,1){20}}

\put(180,40){\makebox(0,0){$\circ$}}
\end{picture}
\end{figure}

There are $l$ copies of $\bullet$ and $l$ copies of $\circ$. The modules constructed in this subsection correspond to those parametrized by $\lambda=[1,0]\in \mathbbm{P}^{1}k$.

\subsection{The indecomposable modules $\widetilde{W}^{2i,j}_{l}$.}
For any positive integer $l$ and $1\leq i \leq \frac{n^{2}}{2},j=0,1$, one has a basis of $\widetilde{W}^{2i,j}_{l}$:
$$\{f_{u}(m)|1\leq m\leq l-1,\;1\leq u\leq n^{2}\}\cup \{f_{u}(n)|1\leq u\leq 2i-1\}\cup \{f_{u}(0)|2i\leq u\leq n^{2}\}$$
with the action given by
\begin{eqnarray*}
&&\k^{-1}\hat{\k}f_{u}(m)=q^{2i-2u}f_{u}(m),\;\;\k\hat{\k}^{\frac{n}{2}}f_{u}(m)=(-1)^{u+j}f_{u}(m),\\
&&E f_{u}(m)=f_{u+1}(m),\\
&&Ff_{u}(m)=(u-1)_{q^{-1}}(1-q^{2i-u})f_{u-1}(m)+\delta_{u,1}f_{n^{2}}(m-1),\end{eqnarray*}
where $f_{n^2+1}(m)=f_{0}(m)=f_{2i}(n)=0$. The diagram of this module is give by

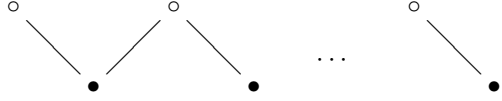
\begin{figure}[hbt]
\begin{picture}(200,50)(0,0)
\put(-30,40){\makebox(0,0){$\circ$}}
 \put(-25,35){\line(1,-1){20}}
\put(0,10){\makebox(0,0){$\bullet$}}
\put(5,15){\line(1,1){20}}
\put(30,40){\makebox(0,0){$\circ$}}
\put(35,35){\line(1,-1){20}}
\put(60,10){\makebox(0,0){$\bullet$}}

\put(90,20){\makebox(0,0){$\cdots$}}
\put(120,40){\makebox(0,0){$\circ$}}
 \put(125,35){\line(1,-1){20}}
\put(150,10){\makebox(0,0){$\bullet$}}
\end{picture}
\end{figure}

There are $l$ copies of $\bullet$ and $l$ copies of $\circ$.The modules constructed in this subsection correspond to those parametrized by $\lambda=[0,1]\in \mathbbm{P}^{1}k$. 

\subsection{The indecomposable modules $T^{2i,j}_{l}(\lambda)$.}
For any positive integer $l$, $1\leq i \leq \frac{n^{2}}{2},j=0,1$ and $\lambda\in k^{\ast}$, the indecomposable modules $T^{2i,j}_{l}(\lambda)$ has a basis:
$$\{a_{u}(m),e_{v}(m)|1\leq m\leq l,\;0\leq u\leq n^{2}-2i,\;1\leq v\leq 2i-1\}$$
with the action given by
\begin{eqnarray*}
&&\k^{-1}\hat{\k}e_{v}(m)=q^{2i-2v}e_{v}(m),\;\;\k\hat{\k}^{\frac{n}{2}}e_{v}(m)=(-1)^{v+j}e_{v}(m),\\
&&E e_{v}(m)=\left \{
\begin{array}{ll} e_{v+1}(m), & \;\;\;\;v\neq 2i-1\\a_{0}(m), &
\;\;\;\;v=2i-1,
\end{array}\right. \\
&&Fe_{v}(m)=\left \{
\begin{array}{ll} (v-1)_{q^{-1}}(1-q^{2i-v})e_{v-1}(m), & \;\;\;\;v\neq 1\\
\lambda a_{n^{2}-2i}(m)+ a_{n^{2}-2i}(m-1),&
\;\;\;\;v=1,
\end{array}\right. \end{eqnarray*}
and
\begin{eqnarray*}
&&\k^{-1}\hat{\k}a_{u}(m)=q^{-2i-2u}a_{u}(m),\;\;\k\hat{\k}^{\frac{n}{2}}a_{u}(m)=(-1)^{u+j}a_{u}(m),\\
&&E a_{u}(m)=\left \{
\begin{array}{ll} a_{u+1}(m), & \;\;\;\;u\neq n^2-2i\\ 0, &
\;\;\;\;u=n^2-2i,
\end{array}\right. \\
&&Fa_{u}(m)=\left \{
\begin{array}{ll} (u+2i-1)_{q^{-1}}(1-q^{-u})a_{u-1}(m), & \;\;\;\;u\neq 0\\
0,& \;\;\;\;u=0,
\end{array}\right.
\end{eqnarray*}
where $a_{u}(-1)=0$ for all $0\leq u\leq n^{2}-2i$. The indecomposable modules
$T^{2i,j}_{l}(\lambda)$ correspond to those parametrized by $\lambda'=[1,\lambda]\in \mathbbm{P}^{1}k$.

\subsection{Auslander-Reiten sequences.}  We have the following Auslander-Reiten sequences:
\begin{gather*} 0\To V_{l+2}^{2i,j}\To V_{l+1}^{2i,j}\oplus V_{l+1}^{2i,j}\To V_{l}^{2i,j}\To 0,\\
0\To \widetilde{V}_{l}^{2i,j}\To \widetilde{V}_{l+1}^{2i,j}\oplus \widetilde{V}_{l+1}^{2i,j}\To \widetilde{V}_{l+2}^{2i,j}\To 0,\\
0\To V_{1}^{2i,j}\To \P_{2i,j}\oplus \S_{n^2-2i+2,j'}\oplus \S_{n^2-2i+2,j'}\To \widetilde{V}_{1}^{2i,j}\To 0,\\
 0\To W_{l}^{2i,j}\To W_{l+1}^{2i,j}\oplus W_{l-1}^{2i,j}\To W_{l}^{2i,j}\To 0,\\
 0\To \widetilde{W}_{l}^{2i,j}\To \widetilde{W}_{l+1}^{2i,j}\oplus \widetilde{W}_{l-1}^{2i,j}\To \widetilde{W}_{l}^{2i,j}\To 0,\\
 0\To W_{l}^{2i,j}(\lambda)\To W_{l+1}^{2i,j}(\lambda)\oplus W_{l-1}^{2i,j}(\lambda)\To W_{l}^{2i,j}(\lambda)\To 0.
\end{gather*}

Comparing with the Auslander-Reiten quiver $\Gamma_{\Lambda}$, we obtain the following:

\begin{theorem}\label{thm3.1} The modules

$\bullet\;\;\P_{2i,j},\;\;\;1\leq i\leq \frac{n^{2}}{2},\;j=0,1;$

$\bullet\;\;V^{2i,j}_l, \;\widetilde{V}^{2i,j}_l\;\;\;1\leq i\leq \frac{n^{2}}{2},\;j=0,1,\;l\geq 0;$

$\bullet\;\;W^{2i,j}_l, \;\widetilde{W}^{2i,j}_l\;\;\;1\leq i\leq \frac{n^{2}}{2},\;j=0,1,\;l\geq 1;$

$\bullet\;\;T^{2i,j}_l(\lambda), \;\;\;1\leq i\leq \frac{n^{2}}{2},\;j=0,1,\;l\geq 0,\;\lambda\in k^{\ast}$

form a complete list of finite-dimensional indecomposable $\Quslz$-modules.
\end{theorem}

\section{Tensor products}

The aim of this section is to give the tensor product decomposition formula for simple modules and projective modules. Looking
at the definition of $\Quslz$, one may find that its comultiplication  is more  complicated than the one of $\uqslz$. However, we still have that this comultiplication preserves some kinds of ``grading",  described in detail as follows:

 Let $\zeta_{2n^{2}}$ be a $2n^{2}$-th primitive root of unity. We
 define an algebraic automorphism $\sigma$ of $\Quslz$ as follows:
 $$\sigma(\k):=\k,\;\;\sigma(\hat{\k}):=\hat{\k},\;\;\sigma(E)=\zeta_{2n^{2}}E,\;\;\sigma(F)=\zeta^{-1}_{2n^{2}}F.$$
 The promised grading is the decomposition of $\Quslz$ into $\sigma$-eigenspaces:
 $$\Quslz=\bigoplus_{s=-(n^{2}-1)}^{n^{2}-1}\Quslz_{s}$$
 where
 $$\Quslz_{s}:=\bigoplus_{b-a=s,0\leq a,b\leq n^{2}-1}F^{a}\mathbf{u}^{0}E^{b}.$$
 
 We say that the elements in $\Quslz_{s}$ have \emph{height} $s$. As usual, the grading on
 $\Quslz\otimes \Quslz$ is given by
 $$\Quslz\otimes \Quslz=
 \bigoplus_{s}(\Quslz\otimes \Quslz)_s:=
 \bigoplus_{u+v=s}\Quslz_{u}\otimes \Quslz_{v}.$$
 The following conclusion is clear.
 \begin{lemma}\label{l4.1} For $1-n^{2}\leq s\leq n^2-1$, we have:
 $$\Delta(\Quslz_{s})\subseteq
 (\Quslz\otimes \Quslz)_s.$$
 \end{lemma}
 
 Recall the definition of $\P_{2i,j}$, the set  $\{E^{l}\alpha_{2i,j},E^{l}\gamma_{2i,j}|0\leq l\leq n^{2}-1\}$ forms a basis of $\P_{2i,j}$.  We already know that $E^{l}\alpha_{2i,j},E^{l}\gamma_{2i,j}$ are $\k^{-1}\hat{\k},\k\hat{\k}^{\frac{n}{2}}$-eigenvectors. Moreover, it is
 not hard to see that $E^{l}\alpha_{2i,j}\in \Quslz_{l+1-n^{2}}$ and $E^{l}\gamma_{2i,j}\in
 \Quslz_{l+2i-n^{2}}$. The same argument can be applied to $\S_{2i,j}$. Therefore,  by regarding $\P_{2i,j}$ and $\S_{2i,j}$ as subspaces of $\Quslz$, one sees that they consist of homogeneous elements and thus they are graded according to the heights.  It follows that, for $1\leq i_{1},i_{2}\leq \frac{n^{2}}{2}$  and $0\leq j_{1},j_{2}\leq 1$, we have:
 $$\P_{2i_{1},j_{1}}\otimes \S_{2i_{2},j_{2}}=\bigoplus_{s}(\P_{2i_{1},j_{1}}\otimes \S_{2i_{2},j_{2}})_{s}.$$
 By Lemma \ref{l4.1}, $E$ (resp. $F$) maps $(\P_{2i_{1},j_{1}}\otimes \S_{2i_{2},j_{2}})_{s}$ to $(\P_{2i_{1},j_{1}}\otimes \S_{2i_{2},j_{2}})_{s+1}$
 (resp. $(\P_{2i_{1},j_{1}}\otimes \S_{2i_{2},j_{2}})_{s-1}$) by left multiplication.

 Let $H$ be a finite-dimensional quasi-Hopf algebra and $M_{1}, M_{2}$ two $H$-modules. It is well-known
 that $M_{1}\otimes M_{2}$ is projective if either $M_{1}$ or $M_{2}$ is projective. This implies that $\P_{2i_{1},j_{1}}\otimes \S_{2i_{2},j_{2}}$
 is a direct sum of indecomposable projective modules. So we have the following \emph{method} to compute $\P_{2i_{1},j_{1}}\otimes \S_{2i_{2},j_{2}}$:
 A $\k^{-1}\hat{\k},\k\hat{\k}^{\frac{n}{2}}$-eigenvector $v$ of lowest height determines an indecomposable projective module $\P_{v}$, which is a summand
 of $\P_{2i_{1},j_{1}}\otimes \S_{2i_{2},j_{2}}$.  We then delete the corresponding $\k^{-1}\hat{\k},\k\hat{\k}^{\frac{n}{2}}$-eigenvectors in $\P_{v}$.
Continue this process  until there is no any other $\k^{-1}\hat{\k},\k\hat{\k}^{\frac{n}{2}}$-eigenvector.

 In the rest of this section, $\overline{i}$ denotes the least non-negative residue of $i$ modulo $2$ and
 $2\P$ stands for $\P\oplus \P$. Using the method we just introduced, we obtain the following decomposition rules.

 \begin{theorem}\label{t4.2} For $1\leq i_{1},i_{2}\leq \frac{n^{2}}{2}$
 and $0\leq j_{1},j_{2}\leq 1$, we have:

 \emph{(a)} If $2i_{1}-1\geq n^{2}-2i_{2}+1$ and $i_{1}\leq i_{2}$, then
 $$\P_{2i_{1},j_{1}}\otimes \S_{2i_{2},j_{2}}\cong \bigoplus_{l=0}^{n^{2}-2i_{2}}\P_{n^{2}+2i_{1}-2i_{2}-2l,\overline{j_{1}+j_{2}+l}}.$$

 \emph{(b)} If $2i_{1}-1\geq n^{2}-2i_{2}+1$ and $i_{1}> i_{2}$, then
\begin{eqnarray*}&&\P_{2i_{1},j_{1}}\otimes \S_{2i_{2},j_{2}}\\
&&\cong \bigoplus_{l=0}^{i_{1}-i_{2}-1}2\P_{2i_{1}-2i_{2}-2l,\overline{j_{1}+j_{2}+l}}\oplus
 \bigoplus_{l=0}^{n^{2}-2i_{1}}\P_{n^{2}-2i_{1}+2i_{2}-2l,\overline{j_{1}+j_{2}+l}}.\end{eqnarray*}

 \emph{(c)} If $2i_{1}-1< n^{2}-2i_{2}+1$ and $i_{1}\leq i_{2}$, then
 \begin{eqnarray*}&&\P_{2i_{1},j_{1}}\otimes \S_{2i_{2},j_{2}}\\
 &&\cong \bigoplus_{l=0}^{2i_{1}-2}\P_{n^{2}+2i_{1}-2i_{2}-2l,\overline{j_{1}+j_{2}+l}}\oplus
 \bigoplus_{l=0}^{\frac{n^{2}}{2}-i_{1}-i_{2}}2\P_{n^{2}-2i_{1}-2i_{2}+2-2l,\overline{j_{1}+j_{2}+l-1}}.\end{eqnarray*}

  \emph{(d)} If $2i_{1}-1< n^{2}-2i_{2}+1$ and $i_{1}> i_{2}$, then
\begin{eqnarray*}&&\P_{2i_{1},j_{1}}\otimes \S_{2i_{2},j_{2}}\\
&&\cong \bigoplus_{l=0}^{i_{1}-i_{2}-1}2\P_{2i_{1}-2i_{2}-2l,\overline{j_{1}+j_{2}+l}}\oplus
\bigoplus_{l=0}^{2i_{2}-2}\P_{n^{2}-2i_{1}+2i_{2}-2l,\overline{j_{1}+j_{2}+l}}\\
&&\oplus \bigoplus_{l=0}^{\frac{n^{2}}{2}-i_{1}-i_{2}}2\P_{n^{2}-2i_{1}-2i_{2}+2-2l,\overline{j_{1}+j_{2}+l-1}}.\end{eqnarray*}
\end{theorem}
\begin{proof} We want to equip every module considered with a formal function that can reflect heights and
 $\k^{-1}\hat{\k},\k\hat{\k}^{\frac{n}{2}}$-eigenvalues. Let

 $y^{?}:\;\;\;\; \textmd{height of a vector;}$\\
 $(q^{-2})^{?}:\;\;\;\; \k^{-1}\hat{\k}-\textmd{eigenvalue of a vector;}$\\
 $(-1)^{?}:\;\;\;\;\k\hat{\k}^{\frac{n}{2}}-\textmd{eigenvalue of a vector}.$ 
 
 Then the formal function $\eta(\S_{2i,j})$ associated to $\S_{2i,j}$ is
 \begin{eqnarray*}
\eta(\S_{2i,j})&=&\sum_{l=0}^{n^{2}-2i}(q^{-2i-2l}y^{l},(-1)^{l}y^{l}).
\end{eqnarray*}
And
$$\eta(\P_{2i,j})=(1+y^{n^{2}})\eta(\S_{n^{2}-2i+2,j'})+2y^{2i-1}\eta(\S_{2i,j}).$$
In case  $2i_{1}-1\geq n^{2}-2i_{2}+1$ and $i_{1}\leq i_{2}$, it is immediate that we have:
$$\eta(\P_{2i_{1},j_{1}})\eta(\S_{2i_{2},j_{2}})=\sum_{l=0}^{n^{2}-2i_{2}}y^{l}\eta(\P_{n^{2}+2i_{1}-2i_{2}-2l,\overline{j_{1}+j_{2}+l}}).$$
This gives the decomposition rules in Part (a).  

For Part (b), we have the equation of the height functions:
\begin{eqnarray*}
\eta(\P_{2i_{1},j_{1}})\eta(\S_{2i_{2},j_{2}})&=&(1+y^{n^{2}})\sum_{l=0}^{i_{1}-i_{2}-1}y^{l}\eta(\P_{2i_{1}-2i_{2}-2l,\overline{j_{1}+j_{2}+l}})\\
&&+\sum_{l=0}^{n^{2}-2i_{1}}y^{2i_{1}-2i_{2}+l}\eta(\P_{n^{2}-2i_{1}+2i_{2}-2l,\overline{j_{1}+j_{2}+l}}).
\end{eqnarray*}

For Part (c), the following equation holds:
\begin{eqnarray*}
\eta(\P_{2i_{1},j_{1}})\eta(\S_{2i_{2},j_{2}})&=&\sum_{l=0}^{2i_{1}-2}y^{l}\eta(\P_{n^2+2i_{1}-2i_{2}-2l,\overline{j_{1}+j_{2}+l}})\\
&&+2\sum_{l=0}^{\frac{n^{2}}{2}-i_{1}-i_{2}}y^{2i_{1}-1+l}\eta(\P_{n^{2}-2i_{1}-2i_{2}+2-2l,\overline{j_{1}+j_{2}+l-1}}).
\end{eqnarray*}

For Part (d), we have:
\begin{eqnarray*}
\eta(\P_{2i_{1},j_{1}})\eta(\S_{2i_{2},j_{2}})&=&(1+y^{n^{2}})\sum_{l=0}^{i_{1}-i_{2}-1}y^{l}\eta(\P_{2i_{1}-2i_{2}-2l, \overline{j_{1}+j_{2}+l}})\\
&&+y^{2i_{1}-2i_{2}}\sum_{l=0}^{2i_{2}-2}y^{l}\eta(\P_{n^2-2i_{1}+2i_{2}-2l,\overline{j_{1}+j_{2}+l}})\\
&&+2y^{2i_{1}-1}\sum_{l=0}^{\frac{n^{2}}{2}-i_{1}-i_{2}}y^{l}\eta(\P_{n^{2}-2i_{1}-2i_{2}+2-2l,\overline{j_{1}+j_{2}+l-1}}).
\end{eqnarray*}
 \end{proof}

 \begin{remark} \emph{Let $H$ be a finite-dimensional quasi-Hopf algebra and assume $0\to M_{1}\to M_{2}\to M_{3}\to 0$ to
 be a short exact sequence of $H$-modules. Tensoring the sequence with a projective $H$-module $P$, we get a split exact sequence:
 $0\to P\otimes M_{1}\to P\otimes M_{2}\to P\otimes M_{3}\to 0$.  Therefore,
 Theorem \ref{t4.2} gives the decomposition formulas for the tensor product of a projective $\Quslz$-module  with any other finite-dimensional $\Quslz$-module. In fact, let $M$ be an arbitrary $\Quslz$-module
 and $\bigoplus_{l} \S_{l}$ the direct sum of its composition factors. Then we have:}
 $$\P\otimes M\cong \bigoplus_{l}\P\otimes \S_{l}.$$
 \end{remark}

 Denote by $K_{0}$ the Grothendieck ring of $\Quslz$. We want to characterize the ring structure of $K_0$.

 \begin{lemma}\label{lem4.4} Let $1\leq i_{1},i_{2}\leq \frac{n^{2}}{2}$
 be two positive integers and $0\leq j_{1},j_{2}\leq 1$. If $2i_{1}-1 \geq n^{2}-2i_{2}+1 $ and $i_{1}\leq i_{2}$, then
 $$\S_{2i_{1},j_{1}}\otimes \S_{2i_{2},j_{2}}\cong \bigoplus_{l=0}^{n^{2}-2i_{2}}\S_{n^{2}+2i_{1}-2i_{2}-2l,\overline{j_{1}+j_{2}+l}}.$$
 \end{lemma}
\begin{proof} The notation introduced in the proof of Theorem \ref{t4.2} will be used freely.
Since $\Soc(\P_{2i_{1},j_{1}})=\S_{2i_{1},j_{1}}$, we have a natural embedding:
$$\S_{2i_{1},j_{1}}\otimes \S_{2i_{2},j_{2}}\hookrightarrow \P_{2i_{1},j_{1}}\otimes \S_{2i_{2},j_{2}}\cong
\bigoplus_{l=0}^{n^{2}-2i_{2}}\P_{n^{2}+2i_{1}-2i_{2}-2l,\overline{j_{1}+j_{2}+l}}.$$
Thanks to the embedding, we have the equation:
\begin{equation}\label{eq;4.1}y^{2i_{1}-1}\eta(\S_{2i_{1},j_{1}})\eta(\S_{2i_{2},j_{2}})=
\sum_{l=0}^{n^2-2i_{2}}y^{l}\eta(\S_{n^{2}+2i_{1}-2i_{2}-2l,\overline{j_{1}+j_{2}+l}}).\end{equation}
It follows that
$$[\S_{2i_{1},j_{1}}\otimes \S_{2i_{2},j_{2}}]=[\bigoplus_{l=0}^{n^{2}-2i_{2}}\S_{n^{2}+2i_{1}-2i_{2}-2l,\overline{j_{1}+j_{2}+l}}]$$
in $K_{0}$. Assume that $\S_{2i_{1},j_{1}}\otimes \S_{2i_{2},j_{2}}=\bigoplus_{t}M_{t}$ is the decomposition of
$\S_{2i_{1},j_{1}}\otimes \S_{2i_{2},j_{2}}$ into indecomposables. So it is enough to show that every $M_{t}$ is simple. Otherwise, there would be
a $M_{t}$ containing at least two simples as composition factors.  Since $M_{t}$ is a submodule of
$\bigoplus_{l=0}^{n^{2}-2i_{2}}\P_{n^{2}+2i_{1}-2i_{2}-2l,\overline{j_{1}+j_{2}+l}}$, $M_{t}$ contains a Verma submodule. But every Verma submodule of $\bigoplus_{l=0}^{n^{2}-2i_{2}}\P_{n^{2}+2i_{1}-2i_{2}-2l,\overline{j_{1}+j_{2}+l}}$ must contain
a $\k^{-1}\hat{\k},\k\hat{\k}^{\frac{n}{2}}$-eigenvector $v$ with height $\leq n^{2}-2i_{2}$. This contradicts to Equation
\eqref{eq;4.1} since the height of any $\k^{-1}\hat{\k},\k\hat{\k}^{\frac{n}{2}}$-eigenvector in $\S_{2i_{1},j_{1}}\otimes \S_{2i_{2},j_{2}}$
 is at least $2i_{1}-1$ which is bigger than $n^{2}-2i_{2}$.
\end{proof}

As a consequence of Lemma \ref{lem4.4}, we obtain the following.

\begin{corollary}\label{c4.7} For $0\leq j_{1},j_{2}\leq 1$ and $2\leq i\leq \frac{n^{2}}{2}-1$, we have:
$$\S_{2i,j_{1}}\otimes \S_{n^{2}-2,j_{2}}\cong \S_{2i+2,\overline{j_{1}+j_{2}}}\oplus \S_{2i,\overline{j_{1}+j_{2}+1}}\oplus
\S_{2i-2,\overline{j_{1}+j_{2}}}.$$
\end{corollary}

\begin{lemma}\label{l4.5}  For $0\leq j_{1},j_{2}\leq 1$, we have:
\begin{eqnarray*}
\S_{n^{2},j_{1}}\otimes \S_{n^{2}-2,j_{2}}&\cong& \S_{n^{2}-2,\overline{j_{1}+j_{2}}}\\
\S_{2,j_{1}}\otimes \S_{n^{2}-2,j_{2}}&\cong &\S_{4,\overline{j_{1}+j_{2}}}\oplus \P_{2,\overline{j_{1}+j_{2}+1}}.
\end{eqnarray*}
\end{lemma}
\begin{proof} The first isomorphism is clear and we prove the second one. Using the same method demonstrated in the proof of Lemma \ref{lem4.4},
we obtain:
$$[\S_{2,j_{1}}\otimes \S_{n^{2}-2,j_{2}}]=[\S_{4,\overline{j_{1}+j_{2}}}]+[\P_{2,\overline{j_{1}+j_{2}+1}}]$$
in $K_{0}$. Since $\S_{4,\overline{j_{1}+j_{2}}}$ and $\P_{2,\overline{j_{1}+j_{2}+1}}$ belong to different blocks,
we know that $\S_{4,\overline{j_{1}+j_{2}}}$ is a direct summand of $\S_{2,j_{1}}\otimes \S_{n^{2}-2,j_{2}}$.
That is,
$$\S_{2,j_{1}}\otimes \S_{n^{2}-2,j_{2}}\cong \S_{4,\overline{j_{1}+j_{2}}}\oplus M$$
with $[M]=[\P_{2,\overline{j_{1}+j_{2}+1}}]$. By Theorem \ref{t4.2} (c), we have:
$$\P_{4,\overline{j_{1}+j_{2}}}\oplus 2\P_{2,\overline{j_{1}+j_{2}+1}}\cong \P_{2,j_{1}}\otimes \S_{n^{2}-2,j_{2}}\twoheadrightarrow
\S_{2,j_{1}}\otimes \S_{n^{2}-2,j_{2}}\cong \S_{4,\overline{j_{1}+j_{2}}}\oplus M.$$
Since $\P_{4,\overline{j_{1}+j_{2}}}$ and $\P_{2,\overline{j_{1}+j_{2}+1}}$ belong to different block, we obtain:
$$2\P_{2,\overline{j_{1}+j_{2}+1}}\twoheadrightarrow M.$$
If $\P_{2,\overline{j_{1}+j_{2}+1}}\not\twoheadrightarrow M$, then $M/JM\cong 2\P_{2,\overline{j_{1}+j_{2}+1}}/J(2\P_{2,\overline{j_{1}+j_{2}+1}})$.
Here $J$ denotes the Jacobson radical of $\Quslz$. Thus
$\dim M/JM=2(n^{2}-1)$ implying
 $\S_{n^{2},\overline{j_{1}+j_{2}}}\subseteq \Soc(M)$.
This is impossible since we also have $\Soc(M)\hookrightarrow 2\P_{2,\overline{j_{1}+j_{2}+1}}$. So
$\P_{2,\overline{j_{1}+j_{2}+1}}\twoheadrightarrow M$. It follows that $\P_{2,\overline{j_{1}+j_{2}+1}}\cong M$ since $\dim \P_{2,\overline{j_{1}+j_{2}+1}}=\dim M$.
\end{proof}

We obtain the following basic observation, which is consistent with the classical $\uqslz$
and $\mathfrak{U}_{q}(\mathfrak{sl}_{2})$ cases.

\begin{proposition}\label{p4.6} Let $1\leq i_{1},i_{2}\leq \frac{n^{2}}{2}$
 be two positive integers and $0\leq j_{1},j_{2}\leq 1$. The indecomposable direct summands of $\S_{2i_{1},j_{1}}\otimes \S_{2i_{2},j_{2}}$
 are either simple or projective.
\end{proposition}
\begin{proof} By Corollary \ref{c4.7}, both $\S_{2i_{1},j_{1}}$ and $\S_{2i_{2},j_{2}}$ occur as direct summands of suitable tensor
powers of $\S_{n^{2}-2,j}$ for $j=0$ or $j=1$. Thus by combining Corollary \ref{c4.7} and Lemma \ref{l4.5}, we get the desired conclusion.
\end{proof}

Combing Theorem \ref{t4.2} with Proposition \ref{p4.6}, we obtain the following refined decomposition rules:
\begin{theorem}\label{thm4.8} For $1\leq i_{1},i_{2}\leq \frac{n^{2}}{2}$
 and $0\leq j_{1},j_{2}\leq 1$, we have:

 \emph{(a)} If $2i_{1}-1\geq n^{2}-2i_{2}+1$ and $i_{1}\leq i_{2}$, then
 $$\S_{2i_{1},j_{1}}\otimes \S_{2i_{2},j_{2}}\cong \bigoplus_{l=0}^{n^{2}-2i_{2}}\S_{n^{2}+2i_{1}-2i_{2}-2l,\overline{j_{1}+j_{2}+l}}.$$

 \emph{(b)} If $2i_{1}-1\geq n^{2}-2i_{2}+1$ and $i_{1}> i_{2}$, then
$$\S_{2i_{1},j_{1}}\otimes \S_{2i_{2},j_{2}}
\cong \bigoplus_{l=0}^{n^{2}-2i_{1}}\S_{n^{2}-2i_{1}+2i_{2}-2l,\overline{j_{1}+j_{2}+l}}.$$

 \emph{(c)} If $2i_{1}-1< n^{2}-2i_{2}+1$ and $i_{1}\leq i_{2}$, then
 \begin{eqnarray*}&&\S_{2i_{1},j_{1}}\otimes \S_{2i_{2},j_{2}}\\
 &&\cong \bigoplus_{l=0}^{2i_{1}-2}\S_{n^{2}+2i_{1}-2i_{2}-2l,\overline{j_{1}+j_{2}+l}}\oplus
 \bigoplus_{l=0}^{\frac{n^{2}}{2}-i_{1}-i_{2}}\P_{n^{2}-2i_{1}-2i_{2}+2-2l,\overline{j_{1}+j_{2}+l-1}}.\end{eqnarray*}

  \emph{(d)} If $2i_{1}-1< n^{2}-2i_{2}+1$ and $i_{1}> i_{2}$, then
\begin{eqnarray*}&&\P_{2i_{1},j_{1}}\otimes \S_{2i_{2},j_{2}}\\
&&\cong \bigoplus_{l=0}^{2i_{2}-2}\S_{n^{2}-2i_{1}+2i_{2}-2l,\overline{j_{1}+j_{2}+l}}
\oplus \bigoplus_{l=0}^{\frac{n^{2}}{2}-i_{1}-i_{2}}\P_{n^{2}-2i_{1}-2i_{2}+2-2l,\overline{j_{1}+j_{2}+l-1}}. 
\end{eqnarray*}
\end{theorem}

Now we are ready to describe the Grothendieck ring $K_{0}$. Let $\mathbbm{Z}[g,x]$ be the polynomial algebra
over $\mathbbm{Z}$  in two variables $g$ and $x$. We define polynomials $f_{2m,j}\in \mathbbm{Z}[g,x]$ inductively for
$0\leq m\leq \frac{n^{2}}{2}-1,\;j=0,1$ as follows:
\begin{eqnarray*}
&&f_{0,0}=1,\;\;f_{0,1}=g,\;\;\;\;f_{2,0}=x,\;\;f_{2,1}=xg,\\
&&f_{4,0}=x^{2}-xg-1,\;\;f_{4,1}=x^{2}g-x-g.\\
\textrm{Assume} \;\;\;\;f_{2m,j}&=&x^{m}g^{j}-a_{1}f_{2(m-1),\overline{j+1}}-a_{2}f_{2(m-2),\overline{j+2}}-\cdots -a_{m}f_{0,\overline{j+m}}. 
\;\; \textrm{Define}\\
f_{2(m+1),j}&=&x^{m+1}g^{j}-(1+a_{1})f_{2m,\overline{j+1}}\\
&&-\sum_{l=1}^{m-1}(a_{l-1}+a_{l}+a_{l+1})f_{2(m-l),\overline{j+l+1}}-
a_{m-1}f_{0,\overline{j+m+1}},
\end{eqnarray*}
where $a_{0}=1$. Let $I$ be the ideal of  $\mathbbm{Z}[g,x]$ generated by $g^{2}-1$ and $f_{n^{2}-2,0}x-2f_{n^{2}-2,1}-f_{n^{2}-4,0}-2$.  We have
the following main result:

\begin{theorem}\label{thm4.9}  The Grothendieck ring $K_{0}$ of $\Quslz$ is isomorphic to the quotient ring
$\mathbbm{Z}[g,x]/I$.
\end{theorem}
\begin{proof} Define a ring map
\begin{eqnarray*}\Upsilon:\;\mathbbm{Z}[g,x]\To K_{0},
&&g\mapsto [\S_{n^{2},1}],\;\; x\mapsto [\S_{n^{2}-2,0}].
\end{eqnarray*}
The map $\Upsilon$ is well-defined since $K_{0}$ is commutative by Theorem \ref{thm4.8}. Using
the first isomorphism in Lemma \ref{l4.5} repeatedly, we obtain:
$$\Upsilon(f_{2m,j})=[\S_{n^{2}-2m,j}],$$
for $0\leq m\leq \frac{n^{2}}{2}-1,\;j=0,1$. It follows that $\Upsilon$ is surjective. Now $\S_{n^{2},1}\otimes \S_{n^{2},1}=\S_{n^{2},0}$ implies that
$\Upsilon(g^{2}-1)=0$. Applying the second isomorphism in Lemma \ref{l4.5}, we obtain $\Upsilon(f_{n^{2}-2,0}x-2f_{n^{2}-2,1}-f_{n^{2}-4,0}-2)=0$.
Therefore, $\Upsilon$ induces an ephimorphism of rings:
$$\bar{\Upsilon}:\;\;\mathbbm{Z}[g,x]/I\twoheadrightarrow  K_{0}.$$
It remains to show that $\Upsilon$ is injective. For convenience, we denote the generators of $\mathbbm{Z}[g,x]/I$  still by $g$ and $x$. Observe that
 $\{[\S_{n^{2},1}]^{j}[\S_{n^{2}-2,0}]^{i}|0\leq j\leq 1, 0\leq i\leq \frac{n^{2}}{2}-1\}$ is a $\mathbbm{Z}$-basis
of $K_{0}$. From this, we can define a $\mathbbm{Z}$-linear map:
\begin{eqnarray*}\Psi:\;K_0\To \mathbbm{Z}[g,x]/I,
&&[\S_{n^{2},1}]^{j}[\S_{n^{2}-2},0]^{i}\mapsto g^{j}x^{i}.
\end{eqnarray*}
It is not hard to check that $\Psi\bar{\Upsilon}=\id$. Hence, $\bar{\Upsilon}$ is injective.
\end{proof}

\section*{Acknowledgments}
The first author would like thank the Department of Mathematics, the University
 of Antwerp for its hospitality during his visiting in 2013.
 The work is supported by the NSF of China (No. 11371186) and a grant in the framework of an FWO project.

\end{document}